\documentclass[a4paper,12pt]{article}
\usepackage{amsmath, amssymb, amsfonts, amsthm,tikz-cd}
\usepackage{mathtools,xcolor}
\usepackage{hyperref}
\usepackage{longtable}
\newtheorem{thm}{Theorem}[section]
\newtheorem{lem}[thm]{Lemma}
\newtheorem{lema}[thm]{Lemma}
\newtheorem{prop}[thm]{Proposition}
\theoremstyle{definition}
\newtheorem{dfn}[thm]{Definition}

\theoremstyle{plain}
\newtheorem{cor}[thm]{Corollary}
\numberwithin{equation}{section}
\usepackage[a4paper, top = 1.2in, bottom = 1.2in, left = 0.8in, right = 0.8in]{geometry}
\numberwithin{equation}{section}

\newcommand{\N}{\mathbb{N}}
\newcommand{\Q}{\mathbb{Q}}

\newcommand{\Z}{\mathbb{Z}}

\newcommand{\Hom}{\mathrm{Hom}}

\def\1{1\!\!1}

\newcommand{\psmat}[4]{\bigl( \begin{smallmatrix} #1 & #2 \\ #3 & #4 \end{smallmatrix} \bigr)}

\title{Infinitely many cubic points on $X_0(N)/\langle w_d\rangle$ with $N$ square-free}

\author{ \ Francesc Bars\footnote{First author is supported by PID2020-116542GB-I00}, Tarun  Dalal }

\begin{document}
\maketitle

\begin{abstract}
We determine all modular curves $X_0(N)/\langle w_d\rangle$ that admit infinitely
many cubic points over the rational field $\mathbb{Q}$, when $N$ is square-free.
\end{abstract}

\section{Introduction}
A non-singular smooth curve $C$ over a number field $K$ of genus
${g_C>1}$ always has a finite set of $K$-rational points $C(K)$ by a
celebrated result of Faltings (here we fix once and for all
$\overline{K}$, an algebraic closure of $K$). We consider the set of
all points of degree at most $n$ for $C$ by
$\Gamma_n(C,K)=\cup_{[L:K]\leq n}C(L)$ and exact degree $n$ by
$\Gamma_n'(C,K)=\cup_{[L:K]=n}C(L)$, where $L\subseteq \overline{K}$
runs over the finite extensions of $K$. A point $P\in C$ is said to
be a point of degree $n$ over $K$  if $[K(P):K]=n.$

The set $\Gamma_n(C,M)$ is infinite for a certain $M/K$ finite
extension if $C$ admits a degree at most $n$ map, {{all
defined over $M$}}, to a projective line or an elliptic curve
{ with positive $M$-rank}. The converse is true for
$n=2$ \cite{HaSi91}, $n=3$ \cite{AH91} and $n=4$ under certain
restrictions \cite{AH91}\cite{DF93}. If we fix the number field $M$
in the above results (i.e. an arithmetic statement for
$\Gamma_n(C,M)$ with $M$ fix), {{we need a precise
understanding over $M$}} of the set $W_n(C)=\{v\in
Pic^n(C)|h^0(C,\mathcal{L}_v)>0\}$ where $Pic^n$ is the usual
$n$-Picard group and $\mathcal{L}_v$ the line bundle of degree $n$
on $C$ associated to $v$. {{If $W_n(C)$} contains no
translates of abelian subvarieties with positive $M$-rank of
$Pic^n(C)$ then $\Gamma_n'(C,M)$ is finite, (under the assumption
that $C$ admits no maps of degree at most $n$ to a projective line
over $M$)}.

For $n=2$ the arithmetic statement {{for
$\Gamma_n(C,K)$}} follows from \cite{AH91}, (for a sketch of the
proof and the precise statement see \cite[Theorem 2.14]{BaMomose}).

For $n=3$, Daeyeol Jeon introduced an arithmetic statement and its
proof in \cite{Jeo21} following \cite{AH91} and \cite{DF93}. In
particular, if $g_C\geq 3$ and $C$ has no degree 3 or 2 map to a
projective line and no degree 2 map to an elliptic curve over
$\overline{K}$  then the set of exact cubic points of $C$ over $K$,
$\Gamma_3'(C,K)$, is an infinite set if and only if $C$ admits a
degree three map to an elliptic curve over $K$ with positive
$K$-rank.

Observe that if {{$g_C\leq 1$ (with $C(K)\neq \emptyset$
for $g_C=1$)}}, then $C$ has a degree three map over $K$ to the
projective line, thus $\Gamma_3'(C,K)$ is always an infinite set.
Thus for curves $C$ with $C(K)\neq \emptyset$ we restrict to
$g_{C}\geq 2$ in order to study the finiteness of the set
$\Gamma_3'(C,K)$. We now introduce a few more notations.

For any $N\in \N$, let $X_0(N)$ denote the modular curve corresponding to the group $\Gamma_0(N)$. 
The modular curve $X_0(N)$ is the coarse moduli space over $\Q$ of the isomorphism classes of the generalized elliptic curves $E$ with a cyclic subgroup $C$ of order $N$.  
For any $d || N$, the $d$-th (partial) Atkin-Lehner operator $w_d$ defined by the action of the matrix $\psmat{dx}{y}{Nu}{dv}$ such that $d^2xv-Nuy=d$. Let $N=d\cdot N^\prime$ where $(d,N^\prime)=1$. Then $w_d$ defines an involution on $X_0(N)$ given by
$$(E,C)\rightarrow (E/A_{N^\prime}, (E_{N^\prime}+A)/A_{N^\prime}),$$
where $E_{N^\prime}$ denotes the kernel of the multiplication by $N^\prime$ and $A_{N^\prime}:= \mathrm{ker}([N^\prime]: A\rightarrow A)$.
Let $X_0^{+d}(N):=X_0(N)/\langle w_d\rangle$ denote the quotient curve. Thus there is a $\Q$-rational degree $2$ mapping $X_0(N)\rightarrow X_0^{+d}(N)$. 
In \cite{BD22}, the authors computed all the values of $N$ for which the curve $X_0(N)/w_N$ has infinitely many cubic points over $\Q$.
In this article, we determine all the values of the pair $(N,d)$ such that the curve $X_0^{+d}(N):=X_0(N)/w_d$ has infinitely many cubic points over the rational field $\Q$ when $N$ is a square free integer and $1<d<N$. Unless otherwise stated explicitly, throughout the article we always assume that $N$ is a square-free positive integer.
Since the the modular curves $X_0^{+d}(N)$ always has a $\Q$-rational points (cusp),
we restrict ourselves to the cases where $g_{X_0^{+d}(N)}:=genus(X_0^{+d}(N))\geq 2$. 
The main result of this article is the following:
\begin{thm} \label{Main Theorem}
Suppose $g_{X_0^{+d}(N)}\geq 2$. Then
$\Gamma_3'(X_0^{+d}(N),\mathbb{Q})$ is infinite if and only if
$g_{X_0^{+d}(N)}=2$ or $(N,d)$ is in the following list:

\begin{longtable}{|c|c|}
\hline
$g_{X_0^{+d}(N)}$&$(N,w_d)$\\
\hline
3&$(42,w_{2}),
(42,w_{7}),
(57,w_{19}),
(58,w_{2}),
(65,w_{5}),
(65,w_{13}),
(77,w_{7}),
(82,w_{41}),$\\
&$(91,w_{13}),
(105,w_{35}),
(118,w_{59}),
(123,w_{41}),
(141,w_{47}),
$\\
\hline
4&$(66 ,w_{33}),
(74 ,w_{37}),
(86 ,w_{43})$\\
\hline
5&$(106,w_{53}),(158,w_{79})$\\
\hline
6&$ (122,w_{61}),(166,w_{83})$\\
\hline
7&$(130,w_{65})$\\
\hline
8&$(178,w_{89})$\\
\hline
9&$(202,w_{101}),(262,w_{131})$\\
\hline
\end{longtable}

\end{thm}

{One of the key idea to decide the set $\Gamma_3'(X, \mathbb{Q})$ is finite or infinite is to check whether there is a $\Q$-rational degree $3$ mapping between $X$ and an elliptic curve $E$ with positive $\Q$-rank. 
Using the existing ideas we can not directly compute the values of $N$ and $d$ such that the set $\Gamma_3'(X_0^{+d}(N), \mathbb{Q})$ is infinite.
For example, with the existing ideas we can not determine whether there is a degree $3$ mapping $X_0(Np)/\langle w_p\rangle\rightarrow E$, where $p$ is a prime, $p\nmid N$ and $\mathrm{Cond}(E)=p$. 
But, in this article, based on the ideas of M.Derick and P.Orli$\acute{c}$ from \cite{DO23}, we develop a criterion to check whether there is a $\Q$-rational degree $3$ mapping $X_0(Np)/\langle w_p\rangle\rightarrow E$, where $p$ is a prime, $p\nmid N$ and $\mathrm{Cond}(E)=p$.}

The \texttt{MAGMA} codes for computing the models of $X_0^{+d}(N)$ can be found at \url{https://github.com/Tarundalalmath/Models-for-X0-N-d} and the \texttt{MAGMA} codes for computing the points of $X_0^{+d}(N)$ over finite fields can be found at \url{https://github.com/FrancescBars/Magma-functions-on-Quotient-Modular-Curves}.


\section{General considerations}
\label{General considerations section}
We recall the notion of gonality of a curve. Given a complete curve $C$ over $K$, the gonality of $C$ is defined as follows:
$$\rm{Gon}(C):=min\{\deg(\varphi)\mid \varphi: C\rightarrow \mathbb{P}^1 \mathrm{defined \ over} \ \overline{K}\}.$$

The following theorem plays a crucial role in deciding whether the set $\Gamma_3'(C,K)$ is infinite or not.
\begin{lem}\cite[Lemma 1.2]{Jeo21} (see also \cite[Lemma 2]{BD22})
\label{main lemma for gonality >=4}
Suppose $C$ has $Gon(C)\geq 4$, { $P\in C(K)$} and does not have a degree $\leq 2$ map to
an elliptic curve. If $\Gamma_3'(C,K)$ is an infinite set then $C$ admits a $K$-rational map of degree $3$ to an elliptic
curve with positive $K$-rank.
\end{lem}

The following results are well known:
\begin{thm}
\begin{enumerate}
\item[(i).](\cite{MaHa99}) The values of $(N,d)$ such that $X_0(N)/\langle w_d\rangle$ is hyperelliptic are given in Table \ref{Hyperelliptic values}.
\item[(ii).] (\cite{HS99}) The values of $(N,d)$ such that $\mathrm{Gon}(X_0(N)/\langle w_d\rangle)=3$ are given in Table \ref{Trigonal values}.
\item[(iii).](\cite{BGK20}) The values of $(N,d)$ such that $X_0(N)/\langle w_d\rangle$ is bielliptic are given in Table \ref{Bielliptic values}.
\end{enumerate}
\end{thm}

We say that a triple $(N,d,E)$, where $N$ is a natural number, $d||N$ and $E$ is an
elliptic curve over $\mathbb{Q}$ with positive $\mathbb{Q}$-rank, is
admissible if there is a $\Q$-rational degree 3 mapping $X_0^{+d}(N)\rightarrow E$.

Throughout this section we consider the values of $N,d$ such that $\mathrm{Gon}(X_0^{+d}(N))\geq 4$ and $X_0^{+d}(N)$ has no degree $\leq 2$ map to an elliptic curve. By Lemma \ref{main lemma for gonality >=4}, the set $\Gamma_3'(X_0^{+d}(N), \Q)$ is infinite if and only if the triple $(N,d,E)$ is admissible for some elliptic curve $E$ with positive $\Q$-rank.

{ The following lemma gives a criterion to rule out the triples which are not admissible (cf. \cite[Lemma 4]{BD22}).}

\begin{lema}\label{sieve}
 If $(N,d,E)$ is admissible, then:
\begin{enumerate}
\item $E$ has conductor $M$ with $M|N$ and for any prime $p\nmid N$ we have
$|\overline{X}_0^{+d}(N)(\mathbb{F}_{p^n})|\leq 3|\overline{E}(\mathbb{F}_{p^n})|$ and
$|\overline{X}_0(N)(\mathbb{F}_{p^n})|\leq 6|\overline{E}(\mathbb{F}_{p^n})|$, $\forall n\in \N$.
\item \label{ii} if conductor of $E$ is $N$, then the degree of the strong Weil
parametrization of $E$ divides 6.
\item \label{iii} for any prime $p\nmid N$ we have
$\frac{p-1}{12}\psi(N)+2^{\omega(N)}\leq 6 (p+1)^2$, where
$\omega(N)$ is the number of prime divisors of $N$ and {$\psi=N \prod_{\substack{q|N\\ q \ prime}}(1+1/q)$} is the
$\psi$-Dedekind function,

\item for any Atkin-Lehner involution $w_r$ of $X_0(N)$ with $r\neq
d$, we have $g_{X_0^{+d}(N)}\leq 3+2\cdot g_{X_0^{+d}(N)/w_r}+2$.

\end{enumerate}
\end{lema}

\begin{cor}
For $N> 623$, the triple $(N,d,E)$ is not admissible.
\end{cor}
\begin{proof}
The proof is similar to that of \cite[Corollary 5]{BD22}.
\end{proof}

We recall the following result (cf. \cite[Lemma 6]{BD22})
{\begin{lema}\label{lem1}
Let $E/\mathbb{Q}$ be an elliptic curve of conductor $N$ and $\varphi: X_0(N)\rightarrow E$ be the strong Weil parametrization of degree $k$ defined over $\mathbb{Q}$.
Suppose
that $w_N$ acts as +1 on the modular form $f_E$ associated to $E$.
Then $\varphi$ factors through $X_0^+(N)$ (and $k$ is even).
\end{lema}
We now make a minor improvement to \cite[Lemma 7]{BD22}.
\begin{lema}\label{descent} Consider a degree {$k$} map $\varphi:X \rightarrow E$ defined over $\Q$,
where $X$ is a quotient modular curve $X_0(N)/W_N$ with $W_N$ a
proper subgroup of $B(N)$ ($B(N)$ is the subgroup of $Aut(X_0(N))$
generated by all Atkin-Lehner involutions). Assume that $cond(E)=M$ $(M|N)$. Let $d\in \mathbb{N}$ with $d||M$, $(d,N/d)=1$ and { $w_d\notin W_N$}, such that $w_d$ acts as $+1$ on the modular form $f_E$ associated to $E$. Then,

\begin{enumerate}
\item if $E$ has no non-trivial $2$-torsion over $\Q$,
then $\varphi$ factors through $X/\langle w_d\rangle$ and {$k$} is even. \label{descent 1}
{\item if $w_d$ has a fixed point on $X$, then $\varphi$ factors through $X/\langle w_d\rangle$ and {$k$} is even. \label{descent 2}} 
\item if $E$ has non-trivial $2$-torsion over
$\Q$ and $k$ is odd, then we obtain a degree
{$k$} map $\varphi':X/\langle w_d\rangle\rightarrow
{E^\prime}$, by taking $w_d$-invariant to $\varphi$,
{where $E^\prime$ is an elliptic curve isogenous to
$E$}.\label{descent 3}

\end{enumerate}
\end{lema}
\begin{proof}
The proofs of (1) and (3) can be found in \cite[Lemma 7]{BD22}. Moreover, from the proof of \cite[Lemma 7]{BD22}, we get $\varphi\circ w_d=\varphi +P$, where $P$ is a 2-torsion point of $E(\mathbb{C})$. Thus it is easy to see that if $w_d$ has a fixed point on $X$, then $P$ is the trivial zero of $E$ and $\varphi$ factors through $X/\langle w_d\rangle$.
\end{proof}

As an immediate consequence of Lemma \ref{descent}(\ref{descent 2}), we obtain the following result:
\begin{cor}
\label{first application of descent 2}
The following triples are not admissible
\begin{align*}
&(130,2,65a1), (130,5,65a1),(130,10,65a1),(130,13,65a1),(130,26,65a1),(185,5,185c1)\\
&(185,37,185c1), (195,3,65a1), (195,5,65a1), (195,13,65a1), (195,39,65a1).
\end{align*}
 
\end{cor}
\begin{proof}
Consider the triple $(130,5,65a1)$. Suppose $f:X_0^{+5}(130)\rightarrow 65a1$ is a $\Q$-rational degree $3$ mapping. By Riemann-Hurwitz theorem, we see that $w_{13}$ has fixed points on $X_0^{+5}(130)$. By Lemma \ref{descent} (\ref{descent 2}), $f$ must factors through $X_0(130)/\langle w_5,w_{13} \rangle$. Which is not possible since $\deg(f)$ is odd. Hence the triple $(130,5,65a1)$ is not admissible. A similar argument will work for the other cases  (for the cases (195,3,65a1) and (195,39,65a1) apply Lemma \ref{descent}(\ref{descent 2}) with the involutions $w_{13}$ and $w_5$ respectively).
\end{proof}

After applying Lemma \ref{sieve}, Lemma \ref{descent}(\ref{descent 1}) and Corollary \ref{first application of descent 2}, we can conclude that if $\mathrm{Gon}(X_0^{+d}(N))\geq 4$ and $X_0^{+d}(N)$ has no degree $\leq 2$ map to an elliptic curve, then the triple $(N,d,E)$ is not admissible possibly only for the following values of $N,d$ and $E$:
\begin{small}
\begin{center}
\begin{tabular}{|c|c|c|}
\hline
N& $d$ &E \\
\hline
106& 53& 53a1\\
\hline
122&61&61a1\\
\hline
129&43&43a1\\
\hline
130&65&65a1\\
\hline
158&79&79a1\\
\hline
166&83&83a1\\
\hline
178&89&89a1\\
\hline
\end{tabular}
\begin{tabular}{|c|c|c|}
\hline
N& $d$ &E \\
\hline
182&91&91b1\\
\hline
183&61&61a1\\
\hline
195&65 &65a1\\
\hline
202&101&101a1\\
\hline
215&43&43a1\\
\hline
222&37&37a1\\
\hline
237&79&79a1\\
\hline
\end{tabular}
\begin{tabular}{|c|c|c|}
\hline
N& $d$ &E \\
\hline
249&83&83a1\\
\hline
262&131&131a1\\
\hline
267&89&89a1\\
\hline
273&91&91b1\\
\hline
303&101&101a1\\
\hline
305&61&61a1\\
\hline
395&79&79a1\\
\hline
\end{tabular}
\end{center}
\end{small}

To handle the remaining triples $(N,d,E)$ mentioned in the last tables, we study the degeneracy maps between quotient modular curves.

\section{Degeneracy maps:} 
Let $M$ be a positive integer such that $M| N$. For every positive divisor $d$ of $\frac{N}{M}$, we have $\iota_d \Gamma_0(N)\iota_d^{-1}\subset \Gamma_0(M)$, where $\iota_d:=\psmat{d}{0}{0}{1}$.
Hence $\iota_d$ induces a degeneracy map
$$\iota_{d,N,M}: X_0(N)\rightarrow X_0(M),$$
where $\iota_{d,N,M}$ acts on $\tau \in \mathbb{H}^*$ in the extended upper half plane as $\iota_{d,N,M}(\tau)=\iota_d\cdot \tau=d\tau$. For any positive integer $r$ with $r||M$, let $w_r^{(N)}$ denote the $r$-th (partial) Atkin-Lehner operator acting on $X_0(N)$ and $w_r^{(M)}$ denote the $r$-th (partial) Atkin-Lehner operator acting on $X_0(M)$. For any positive divisor $d$ of $\frac{N}{M}$ with $(d,r)=1$, a simple calculation shows that $\iota_d w_r^{(N)}\iota_d^{-1}=w_r^{(M)}$. Thus $\iota_{d,N,M}$ induces a mapping 
$$\iota_{d,N,M,r}: X_0(N)/\langle w_r^{(N)}\rangle\rightarrow X_0(M)/\langle w_r^{(M)}\rangle.$$
 Moreover, we have the following commutative diagram
\begin{equation}
\begin{tikzcd}
&X_0(N)\arrow[r,"\iota_{d,N,M}"''] \arrow[d, "w_r^{(N)}"'] &X_{0}(M)\arrow[d, "w_r^{(M)}"'']\\
&X_0(N)/\langle w_r^{(N)} \rangle \arrow[r, "\iota_{d,N,M,r}"] &X_0(M)/\langle w_r^{(M)}\rangle .
\end{tikzcd}
\end{equation}
By abuse of notations we write $w_r$ for $w_r^{(N)}$ and $w_r^{(M)}$. As an immediate consequence we obtain:
\begin{cor}
\label{Induced map corollary}
The modular curves $X_0^{+53}(106), X_0^{+61}(122), X_0^{+65}(130), X_0^{+79}(158)$, $X_0^{+83}(166)$, $X_0^{+89}(178), X_0^{+101}(202), X_0^{+131}(262)$ have infinitely many cubic points over $\Q$.
\end{cor}
\begin{proof}
By the previous discussions, there are natural $\Q$-rational degree $3$ mappings
\begin{align*}
X_0(106)/w_{53}&\rightarrow X_0(53)/w_{53}\cong 53a1\\
X_0(122)/w_{61}&\rightarrow X_0(61)/w_{61}\cong 61a1\\
X_0(130)/w_{65}&\rightarrow X_0(65)/w_{65}\cong 65a1\\
X_0(158)/w_{79}&\rightarrow X_0(79)/w_{79}\cong 79a1\\
X_0(166)/w_{83}&\rightarrow X_0(83)/w_{83}\cong 83a1\\
X_0(178)/w_{89}&\rightarrow X_0(89)/w_{89}\cong 89a1\\
X_0(202)/w_{101}&\rightarrow X_0(101)/w_{101}\cong 101a1\\
X_0(262)/w_{131}&\rightarrow X_0(131)/w_{131}\cong 131a1.
\end{align*}
Since all the elliptic curves $53a1,61a1,65a1,79a1,83a1,89a1,101a1$ and $131a1$ have positive $\Q$-ranks, we conclude that all the modular curves mentioned in the statement of this corollary have infinitely many cubic points over $\Q$.
\end{proof}

We now introduce some notations.
For any curves $C,C^\prime$ over a field $k$, and any morphism $f:C\rightarrow C^\prime$ we define the following mappings
$$f_*:J(C)\rightarrow J(C^\prime)\ \mathrm{defiend\ by} \ f_*([\sum n_iP_i])=[\sum n_i f(P_i)], \ \mathrm{and}$$
$$f^*:J(C)\rightarrow J(C^\prime)\ \mathrm{defined\ by} \ f^*([\sum n_iP_i])=[\sum n_i f^{-1}(P_i)],$$
where $J(C)$ (resp., $J(C^\prime)$) denote the Jacobian of $C$ (resp., $C^\prime$). It is easy to check that $f_*\circ f^*=[deg(f)]$.

For any positive integer $M$, we use the notation $[E,r]_M$ to denote a pair $(E,r)$ where $E$ is an elliptic curve of conductor $M$ and $r||M$ is a positive integer such that $w_r$ acts as $+1$ on $E$, furthermore if $r<M$, then at least one of the following is true:
\begin{enumerate}
\item $E$ has no non-trivial $2$-torsion point,
\item $w_r$ has a fixed point on $X_0(M)$.
\end{enumerate}
For a pair $[E,r]_M$, if $f: X_0(M)\rightarrow E$ is the modular parametrization,  then by Lemma \ref{lem1} and Lemma \ref{descent}, $f$ induces a mapping 
\begin{equation}
\label{induced modular parametrization for quotient curve}
f_r: X_0(M)/\langle w_r\rangle \rightarrow E,
\end{equation}
and the following diagram commutes:
\begin{equation}
\label{level M modular curve commutative diagram with r-th Atkin-Lehner operator}
\begin{tikzcd}
&X_0(M)\arrow[r,"f"''] \arrow[d, "w_r"'] &E\\
&X_0(M)/\langle w_r\rangle \arrow[ur, "f_r"'].
\end{tikzcd}
\end{equation}
Observe that if $E^\prime$ is any elliptic curve isogenous to $E$ and $f^\prime: X_0(M)/\langle w_r\rangle\rightarrow E^\prime$ be a mapping, then $f^\prime$ factors through $f_r$.
Also, for any positive divisor $d$ of $\frac{N}{M}$ with $(d,r)=1$, we have the following commutative diagram:
\begin{equation}
\label{commutative diagram with degeneracy map for elliptic curve and modular curve}
\begin{tikzcd}
&X_0(N)\arrow[r,"\iota_{d,N,M}"''] \arrow[d, "w_r^{(N)}"'] &X_{0}(M)\arrow[d, "w_r^{(M)}"']\arrow[r,"f"''] &E\\
&X_0(N)/\langle w_r\rangle \arrow[r, "\iota_{d,N,M,r}"] &X_0(M)/\langle w_r\rangle\arrow[ur, "f_r"'] .
\end{tikzcd}
\end{equation}
Let $J_0(N)^{w_r}$ (resp., $J_0(M)^{w_r}$) denote the Jacobian of $X_0(N)/\langle w_r\rangle$ (resp., $X_0(M)/\langle w_r\rangle$) and $n$ be the number of divisors of $\frac{N}{M}$. Considering the pushforward maps of \eqref{commutative diagram with degeneracy map for elliptic curve and modular curve}, we obtain the commutative diagram:
\begin{equation}
\label{commutative diagram for Jacobians with degeneracy map for elliptic curve and modular curve}
\begin{tikzcd}
&J_0(N)\arrow[r,"\iota_{d,N,M,*}"''] \arrow[d, "w_{r,*}^{(N)}"'] &J_{0}(M)\arrow[d, "w_{r,*}^{(M)}"']\arrow[r,"f_*"''] &E\\
&J_0(N)^{w_r} \arrow[r, "\iota_{d,N,M,r,*}"] &J_0(M)^{w_r}\arrow[ur, "f_{r,*}"']. 
\end{tikzcd} 
\end{equation}
Note that there is a natural mapping $\delta_{N,M,r}:J_0(N)^{w_r}\rightarrow (J_0(M)^{w_r})^n$ defined by
$$\delta_{N,M,r}:=(\iota_{1,N,M,r,*},\ldots, \iota_{\frac{N}{M},N,M,r,*}).$$
We recall the notion of optimal $B$-isogenous quotient (cf. \cite[Definition 3.2]{DO23}).
\begin{dfn}
    Let $A,B$ be abelian varieties over a field $k$ with $B$ simple. An abelian variety $A^\prime$ together with a quotient map $\pi: A\rightarrow A^\prime$ is an optimal $B$-isogenous quotient if $A^\prime$ is isogenous to $B^n$ for some integer $n$ and every morphism $A\rightarrow B^\prime$ with $B^\prime$ isogenous to $B^m$ for some integer $m$ uniquely factors through $\pi$.
\end{dfn}

{
\begin{prop}
Suppose $N<408$ is a positive integer. For a positive divisor $M$ of $N$ consider a pair $[E,r]_M$, where $E$ is a strong Weil curve over $\Q$ of positive rank and conductor $M$.
Let $f_r: X_0(M)/\langle w_r\rangle \rightarrow E$ be the mapping induced by the modular parametrization $f: X_0(M)\rightarrow E$. 
Consider the mapping $\xi_{E,N,r}:=(f_{r,*})^n\circ \delta_{N,M,r}: J_0(N)^{w_r}\rightarrow E^n$, where $n$ denotes the number of divisors of $\frac{N}{M}$.
 Then $\xi_{E,N,r}^\lor: E^n\rightarrow J_0(N)^{w_r}$ has a trivial kernel (where $\xi_{E,N,r}^\lor$ denotes the dual mapping of $\xi_{E,N,r}$). Hence $\xi_{E,N,r}^\lor: E^n\rightarrow J_0(N)^{w_r}$ is a maximal $E$-isogenous abelian subvariety and $\xi_{E,N,r}: J_0(N)^{w_r}\rightarrow E^n$ is an optimal $E$-isogenous quotient of $J_0(N)^{w_r}$.
\end{prop}
\begin{proof}
Consider the mapping $\xi_{E,N}:=(f_*)^n\circ (\iota_{1,N,M,*},\ldots, \iota_{\frac{N}{M},N,M,*}): J_0(N)\rightarrow E^n$ (cf. \cite[Definition 3.8]{DO23}).
 From the commutative diagram \eqref{commutative diagram for Jacobians with degeneracy map for elliptic curve and modular curve} we get
\begin{equation}
\xi_{E,N}=\xi_{E,N,r}\circ w^{(N)}_{r,*}.
\end{equation}
Hence $\xi_{E,N}^\lor=(w_{r,*}^{(N)})^\lor \circ \xi_{E,N,r}^\lor$. Since $\xi_{E,N}^\lor$ has trivial kernel (cf. \cite[Proposition 3.9]{DO23}), we conclude that $\xi_{E,N,r}^\lor$ has trivial kernel. The remaining statements follows from \cite[Proposition 3.9]{DO23} and the commutative diagram \eqref{commutative diagram for Jacobians with degeneracy map for elliptic curve and modular curve}.
\end{proof}
}
Let $E$ be a strong Weil curve of conductor $M$ ($M| N$) and odd $\Q$-rank, with $w_M$ acts as $+1$ on $E$, and $f_M:X_0(M)/\langle w_M\rangle\rightarrow E$  be the mapping induced by the modular parametrization of $E$. Since $\xi_{E,r,N}:J_0(N)^{w_M}\rightarrow E^n$ is an $E$-isogenous optimal quotient when $N<408$, every mapping  $J_0(N)^{w_M}\rightarrow E$ uniquely factors through $E^n$. Therefore we get that the maps $f_M\circ d_i$, where $d_i$ runs over the degeneracy maps $X_0(N)/\langle w_M\rangle\rightarrow X_0(M)/\langle w_M\rangle$, form a basis for $\mathrm{Hom}_\Q(J_0(N)^{w_M},E)\cong \mathrm{Hom}_\Q(E^n,E)\cong \Z^n$.


\subsection{Degree Pairing} We recall the definition of degree pairing from \cite{DO23}.
\begin{dfn}\cite[Definition 2.2]{DO23}
Let $C,E$ be curves over a field $k$, where $E$ is an elliptic curve. The degree pairing is defined on $\Hom(C,E)$ as
\begin{align*}
\langle \mbox{}\char`\_,\char`\_\mbox{}\rangle: \Hom(C,E)\times \Hom(C,E) &\rightarrow 
\mathrm{End}(Jac(E))\\
f,g &\rightarrow f_*\circ g^*.
\end{align*}
Note that with this notation we have $\langle f,f \rangle=[\deg (f)]$ for $f\in \Hom(C,E)$. If $P\in C(k)$, then we can define the degree pairing on $\Hom(J(C),E)$ as follows
\begin{align*}
\langle \mbox{}\char`\_,\char`\_\mbox{}\rangle: \Hom(J(C),E)\times \Hom(J(C),E) &\rightarrow 
\mathrm{End}(J(E))\\
f,g &\rightarrow (f\circ f_P)_*\circ (g\circ g_P)^*,
\end{align*}
where the mapping $f_P$ (similar for $g_P$) is defined as
\begin{align*}
f_P: C&\rightarrow J(C)\\
x&\rightarrow [x-P].
\end{align*}
\end{dfn}

\begin{prop}\label{computing th degree pairing}
Assume that $N$ is a square free positive integer and $M|N$.
 Let $E$ be an elliptic curve of conductor $M$ with newform $\sum_{n=1}^{\infty} a_nq^n$ such that $w_M^{(M)}$ acts as $+1$ on $E$. Suppose that $f_M:X_0(M)/\langle w_M^{(M)}\rangle\rightarrow E$ be the mapping induced by the modular parametrization $f:X_0(M)\rightarrow E$. For any positive divisors $d_1,d_2$ of $\frac{N}{M}$ we define $a:=a_{\frac{d_1d_2}{gcd(d_1,d_2)^2}}$, then we have the following equality in $\mathrm{End}(Jac(E))$:         
     $$[2]\langle f_M\circ \iota_{d_1,N,M,M},f_M\circ 
\iota_{d_2,N,M,M}\rangle=
               [2][deg(f_M)]\Big[a\frac{\psi(N)}{\psi(\frac{M \cdot lcm(d_1,d_2)}{gcd(d_1,d_2)})}\Big],$$
               where $\psi(s)=s\prod_{p|s}(1+\frac{1}{p})$.
\end{prop}
%
\begin{proof}
     By the definition of the pairing, we have
     \begin{align*}
      \langle f_M\circ \iota_{d_1,N,M,M}\circ 
w_M^{(N)},f_M\circ\iota_{d_2,N,M,M}&\circ w_M^{(N)} \rangle=
                   (f_M\circ \iota_{d_1,N,M,M}\circ w_M^{(N)})_* \circ 
(f_M\circ \iota_{d_2,N,M,M}\circ w_M^{(N)})^*\\
&= (f_M)_*\circ
                   (\iota_{d_1,N,M,M})_* \circ (w_M^{(N)})_*\circ 
(w_M^{(N)})^*\circ (\iota_{d_2,N,M,M})^* \circ (f_M)^*\\
&= (f_M)_*\circ
                   (\iota_{d_1,N,M,M})_*\circ  [deg(w_M^{(N)})] \circ 
(\iota_{d_2,N,M,M})^* \circ (f_M)^*\\
&= (f_M)_*\circ
                   (\iota_{d_1,N,M,M})_*\circ  [2] \circ 
(\iota_{d_2,N,M,M})^* \circ (f_M)^*\\
&= [2]\circ (f_M)_*\circ
                   (\iota_{d_1,N,M,M})_*\circ   
(\iota_{d_2,N,M,M})^* \circ (f_M)^*\\
&= [2]\circ \langle f_M\circ 
\iota_{d_1,N,M,M},f_M\circ\iota_{d_2,N,M,M} \rangle
     \end{align*}
     On the other hand, from the diagram \eqref{commutative diagram with degeneracy map for elliptic curve and modular curve} we have the following equalities inside $\mathrm{End}(E)$:
     \begin{align*}
         \langle f_M\circ \iota_{d_1,N,M,M}\circ w_{M}^{(N)},f_M\circ 
\iota_{d_2,N,M,M}\circ w_{M}^{(N)} \rangle &= \langle f_M\circ w_M^{(M)}\circ \iota_{d_1,N,M},f_M\circ 
w_{M}^{(M)} \circ \iota_{d_2,N,M} \rangle\\
&=\langle f\circ \iota_{d_1,N,M},f\circ 
\iota_{d_2,N,M}\rangle.
     \end{align*}
     By \cite[Theorem 2.13]{DO23}, we have 
     \begin{equation*}
         \langle f\circ \iota_{d_1,N,M},f\circ \iota_{d_2,N,M}\rangle = \Big[a\frac{\psi(N)}{\psi(\frac{M \cdot lcm(d_1,d_2)}{gcd(d_1,d_2)})}deg(f)\Big]= \Big[a\frac{\psi(N)}{\psi(\frac{M \cdot lcm(d_1,d_2)}{gcd(d_1,d_2)})}deg(f_M)\Big][2]
     \end{equation*}
     Therefore we conclude that 
     \begin{equation*}
         [2]\circ \langle f_M\circ 
\iota_{d_1,N,M,M},f_M\circ\iota_{d_2,N,M,M} \rangle=\langle f\circ \iota_{d_1,N,M},f\circ 
\iota_{d_2,N,M}\rangle=[2][deg(f_M)]\Big[a\frac{\psi(N)}{\psi(\frac{M \cdot lcm(d_1,d_2)}{gcd(d_1,d_2)})}\Big].
     \end{equation*}
     This completes the proof.
\end{proof}

We now explain via some examples how to use the above recipe to study the triples $(N,d,E)$.

\section{The remaining cases from \S\ref{General considerations section}}
\label{Quadratic forms for the remaining cases section}
In this section we give complete explanation the cases $(222,37,37a1)$ and $(195,65,65a1)$. The arguments for the other cases are similar to these.
\subsection{The curve $X_0(222)/\langle w_{37}\rangle$:}
We want to check whether there is a $\Q$-rational degree $3$ mapping $X_0(222)/\langle w_{37}\rangle\rightarrow 37a1$. The modular parametrization $f:X_0(37)\rightarrow 37a1$ has degree $2$ and $w_{37}$ acts as $+1$ on $37a1$. By Lemma \ref{descent}, $f$ factors through $X_0(37)/\langle w_{37}\rangle$ and the induced map $f_{37}:X_0(37)/\langle w_{37}\rangle \rightarrow 37a1$ has degree 1 (note that $X_0(37)/\langle w_{37}\rangle \cong 37a1$). The degeneracy maps are $\iota_{1,222,37,37}, \iota_{2,222,37,37}, \iota_{3,222,37,37}$ and $\iota_{6,222,37,37}$ (note that all these maps has degree $12$). Moreover, by Proposition \ref{computing th degree pairing}, we have
\begin{align*}
[2]\langle f_{37}\circ \iota_{1,222,37,37}, f_{37}\circ \iota_{2,222,37,37} \rangle&= [2][a_2. \psi(3).deg(f_{37})]=[2][-2.4.1]=[-16],\\
[2]\langle f_{37}\circ \iota_{1,222,37,37}, f_{37}\circ \iota_{3,222,37,37} \rangle&= [2][a_3. \psi(2).deg(f_{37})]=[2][-3.3.1]=[-18],\\
[2]\langle f_{37}\circ \iota_{1,222,37,37}, f_{37}\circ \iota_{6,222,37,37} \rangle&= [2][a_6. \psi(1).deg(f_{37})]=[2][6.1.1]=[12],\\
[2]\langle f_{37}\circ \iota_{2,222,37,37}, f_{37}\circ \iota_{3,222,37,37} \rangle&= [2][a_6. \psi(1).deg(f_{37})]=[2][6.1.1]=[12],\\
[2]\langle f_{37}\circ \iota_{2,222,37,37}, f_{37}\circ \iota_{6,222,37,37} \rangle&= [2][a_3. \psi(2).deg(f_{37})]=[2][-3.3.1]=[-18],\\
[2]\langle f_{37}\circ \iota_{3,222,37,37}, f_{37}\circ \iota_{6,222,37,37} \rangle&= [2][a_2. \psi(3).deg(f_{37})]=[2][-2.4.1]=[-16].
\end{align*}
The maps $f_{37}\circ \iota_{1,222,37,37}, f_{37}\circ \iota_{2,222,37,37}, f_{37}\circ \iota_{3,222,37,37}$ and $f_{37}\circ \iota_{6,222,37,37}$ form a basis for $\Hom_\Q(J_0(222)^{w_{37}},37a1)$. If $\varphi: X_0(222)/\langle w_{37}\rangle\rightarrow 37a1$ is a $\Q$-rational mapping, then we can write
\begin{equation}
\label{basis combination for 222 mod 37}
\varphi= x_1(f_{37}\circ \iota_{1,222,37,37})+ x_2(f_{37}\circ \iota_{2,222,37,37})+ x_3(f_{37}\circ \iota_{3,222,37,37}) + x_4(f_{37}\circ \iota_{6,222,37,37}),
\end{equation}
for some $x_1,x_2,x_3,x_4\in \Z$. From \eqref{basis combination for 222 mod 37} we have
\begin{align*}
[\deg \varphi]=[12x_1^2]+[12x_2^2]+[12x_3^2]+ &[12x_4^2]+[-16x_1x_2]+[-18x_1x_3]+[12x_1x_4]+[12x_2x_3]\\
&+[-18x_2x_4]+[-16x_3x_4],
\end{align*}
where $``[a]"$ denotes the multiplication by $a$-mapping on the elliptic curve $E=37a1$. Hence we must have
\begin{equation}
\deg\varphi= 12x_1^2+12x_2^2+12x_3^2+12x_4^2 -16x_1x_2-18x_1x_3+12x_1x_4+12x_2x_3-18x_2x_4-16x_3x_4.
\end{equation}
This shows that $\deg \varphi$ is of the form $2g(x_1,x_2,x_3,x_4)$, so it can not take the value $3$. Consequently, the triple $(222,37,37a1)$ is not admissible.

\subsection{The curve $X_0(195)/w_{65}$:} We want to check whether there is a $\Q$-rational degree $3$ mapping $X_0(195)/w_{65}\rightarrow 65a1$. The modular parametrization $f:X_0(65)\rightarrow 65a1$ has degree 2. By Lemma \ref{descent}, $f$ factors through $X_0(65)/\langle w_{65}\rangle$ and the induced map $f_{65}:X_0(65)/w_{65}\rightarrow 65a1$ has degree 1, since $X_0(65)/w_{65}\cong 65a1$. Both the degeneracy maps $\iota_{1,195,65,65}$ and $\iota_{3,195,65,65}$ has degree $4$. Moreover, we have 
$$[2]\langle f_{65}\circ \iota_{1,195,65,65}, f_{65}\circ \iota_{3,195,65,65} \rangle= [2][a_3. \frac{\psi(195)}{\psi(65.3)} .deg(f_{65})]=[2][-2.1.1]=[-4].$$
The maps $f_{65}\circ \iota_{1,195,65,65}, f_{65}\circ \iota_{3,195,65,65}$ form a basis for $\Hom_\Q(J_0(195)^{w_{65}},65a1)$. If $\varphi: X_0(195)/w_{65}\rightarrow 65a1$ is $\Q$-rational mapping then we can write
$$\varphi=x_1(f_{65}\circ \iota_{1,195,65,65})+ x_2(f_{65}\circ \iota_{3,195,65,65}).$$
Thus we have
\begin{align*}
[\deg(\varphi)]&=[x_1^2\deg(f_{65}\circ \iota_{1,195,65,65})] + [2x_1x_2] \langle f_{65}\circ \iota_{1,195,65,65}, f_{65}\circ \iota_{3,195,65,65} \rangle  +[x_2^2\deg(f_{65}\circ \iota_{3,195,65,65})],\\
&=[4x_1^2]+[x_1x_2][-4]+[4x_2^2],
\end{align*}
where $``[a]"$ denotes the multiplication by $a$-mapping on the elliptic curve $E=65a1$.

Hence we must have 
\begin{equation}
\label{degree binary form for 195 mod 65}
\deg(\varphi)=4x_1^2-4x_1x_2+4x_2^2.
\end{equation}
From \eqref{degree binary form for 195 mod 65}, we see that $\deg(\varphi)$ can not take the value $3$.
Consequently, $\Hom_\Q(J_0(195)^{w_{65}},65a1)$ has no element of order 3. Thus there is no $\Q$-rational degree $3$-mapping $X_0(195)/w_{65}\rightarrow 65a1$.

The quadratic forms for the other remaining cases are given in the following table. It is clear from the following table that the triples $(N,d,E)$ appearing in the following table are not admissible. Consequently, the curves $X_0^{+d}(N)$ has finitely many cubic points over $\Q$ for $N,d$ appearing in the following table.
\begin{center}
\begin{longtable}{|c|c|c|c|c|c|}
\hline
N& $d$ &E & g($X_0^{+d}(N)$) & Quadratic Form \\
\hline
129 & 43 & 43a1&7& $4x_1^2-4x_1x_2+4x_2^2$ \\
\hline
182 & 91 & 91b1&10&$6x_1^2+6x_2^2$\\
\hline
183 & 61 & 61a1&10&$4x_1^2-4x_1x_2+4x_2^2$\\
\hline
195 & 65 & 65a1 &9&$4x_1^2-4x_1x_2+4x_2^2$ \\
\hline
215&43&43a1&11 &$6x_1^2-8x_1x_2+6x_2^2$\\
\hline
222 & 37 &37a1&18&$2\cdot g(x_1,x_2,x_3,x_4)$\\
\hline
237&79&79a1&13&$4x_1^2-2x_1x_2+4x_2^2$\\
\hline
249 & 83 & 83a1&8&$4x_1^2-2x_1x_2+4x_2^2$ \\
\hline
267 & 89 & 89a1&9&$4x_1^2-2x_1x_2+4x_2^2$\\
\hline
273&91&91b1&17&$8x_1^2-8x_1x_2+8x_2^2$\\
\hline
303 & 101 & 101a1&10&$4x_1^2-4x_1x_2+4x_2^2$ \\
\hline
305&61&61a1&12&$6x_1^2-6x_1x_2+6x_2^2$\\
\hline
395 & 79 & 79a1&15 &$6x_1^2-6x_1x_2+6x_2^2$\\
\hline
\caption{Quadratic forms for remaining cases}
\label{second remaining cases}
\end{longtable}
\end{center}
\section{Hyperelliptic cases}

In this section we deal with the pairs $(N,w_d)$ such that $X_0(N)/\langle w_d\rangle$ is hyperelliptic. We first consider the hyperelliptic curves of genus $2$.
\begin{thm}
The curve $X_0(N)/\langle w_d\rangle$ has infinitely many cubic points over $\Q$ for the following pairs of $(N,w_d)$:
\begin{align*}
&(30,w_2),(30,w_3),(30,w_{10}),(33,w_3),(35,w_7),(39,w_{13}),(42,w_3),
(42,w_6),(42,w_{21}),(57,w_3),\\
&(58,w_{29}),(66,w_{11}),(70,w_{35}),(78,w_{39}),
(142,w_{71}).
\end{align*}
\end{thm}
\begin{proof}
Using the MAGMA code ``\texttt{\#Points(SimplifiedModel(X0NQuotient(N,[d])):Bound:=10)}", we see that for the above mentioned pairs of $(N,w_d)$, the curve $X_0(N)/\langle w_d \rangle$ has atleast three rational points over $\Q$. From \cite[Lemma 2.1]{JKS04}, we conclude that for each of these cases there is a $\Q$-rational degree $3$ mapping $X_0(N)/\langle w_d\rangle\rightarrow \mathbb{P}^1$. Consequently, $X_0(N)/\langle w_d\rangle$ has infinitely many cubic points over $\Q$.
\end{proof}
The remaining genus two cases are $X_0(38)/\langle w_2\rangle$ and $X_0(87)/\langle w_{29}\rangle$.
\begin{thm}
Both the sets $\Gamma_3'(X_0(38)/\langle w_2\rangle, \Q)$ and $\Gamma_3'(X_0(87)/\langle w_{29}\rangle, \Q)$ are infinite.
\end{thm}
\begin{proof}
An affine model of $X_0(38)/\langle w_2 \rangle$ is given by
\begin{equation}
X_0(38)/\langle w_2\rangle: y^2 = x^6 - 4x^5 - 6x^4 + 4x^3 - 19x^2 + 4x- 12.
\end{equation}
It is easy to see that $X_0(38)/\langle w_2\rangle$ has two $\Q$-rational points $(1:1:0),(1:-1:0)$ and the hyperelliptic involution permutes these points. From \cite[Lemma 2.2]{Jeo21}, we conclude that there is a $\Q$-rational degree $3$ mapping $X_0(38)/\langle w_2\rangle\rightarrow \mathbb{P}^1$. Consequently the set $\Gamma_3'(X_0(38)/\langle w_2\rangle, \Q)$ is infinite. A similar argument works for the curve $X_0(87)/\langle w_{29}\rangle$, which has an affine model $y^2 = x^6 - 2x^4 - 6x^3 - 11x^2 - 6x - 3$.
\end{proof}

Now consider the pairs $(N,w_d)$ such that $X_0(N)/\langle w_d\rangle$ is hyperelliptic and $g(X_0(N)/\langle w_d\rangle)\geq 3$. If the set  $\Gamma_3'(X_0(N)/\langle w_d\rangle, \Q)$ is infinite, then
$W_3(X_0(N)/\langle w_d\rangle)$ must contain an elliptic curve with positive
$\mathbb{Q}$-rank (cf \cite[\S 2.3]{Jeo21}).

Thus, by Cremona tables \cite{Cre} we obtain (because there is no elliptic
curve with positive $\Q$-rank for levels dividing $N$):
\begin{thm}
The set $\Gamma_3'(X_0(N)/\langle w_d\rangle, \Q)$ is finite for the following pairs of $(N,w_d)$:
\begin{align*}
&(46,w_2),(51,w_3),(55,w_5),(70,w_{14}),(78,w_{26}),(95,w_{19}),(62,w_2),(66,w_6),(69,w_3),(70,w_{10}),\\
&(119,w_{17}), (87,w_3),(95,w_5),(78,w_6),(94,w_2),(119,w_7).
\end{align*}
\end{thm}
\section{Trigonal cases}

It is well known that, 
if $C/\Q$ is a trigonal curve of genus $3$ with a $\Q$-rational point, then the projection from the $\Q$-rational point defines a degree $3$ mapping $C\rightarrow \mathbb{P}^1$ over $\Q$. Consequently, in these cases the set $\Gamma_3^\prime(C,\Q)$ is infinite. 

Now consider the pairs $(N,w_d)$ such that $X_0(N)/\langle w_d\rangle$ is trigonal and $g(X_0(N)/\langle w_d\rangle)=4$. A model of $X_0(N)/\langle w_d \rangle$ can be constructed using Petri's theorem. It is well known that a non hyperelliptic curve $X$ (defined over $\Q$) of genus $4$ lies either on a ruled surface or on a quadratic cone (defined over either $\Q$, a quadratic field or a biquadratic field) (cf. \cite[Page 131]{HaS99}). If the ruled surface or the quadratic cone is defined over $\Q$, then there is a degree $3$ mapping $X\rightarrow \mathbb{P}^1$ defined over $\Q$. For example, consider the curve $X_0(70)/\langle w_5\rangle$. Choosing the following basis of weight $2$ cusp forms $S_2(\langle \Gamma_0(70),w_5)$,
\begin{align*}
q + q^8 - 2q^9 - q^{10} + q^{11} + O(q^{12})\\
q^2 + q^6 - 2q^8 - q^{10} + O(q^{12})\\
q^3 - 3q^5 - 2q^6 - q^7 + 3q^8 - 2q^9 + 3q^{10} + 2q^{11} + O(q^{12})\\
q^4 - q^5 - q^6 - q^7 + 2q^8 - q^9 + q^{10} + 3q^{11} + O(q^{12}),
\end{align*}
and using \texttt{MAGMA}, a model of $X_0(70)/\langle w_5\rangle$ is given by 
$$
\begin{cases}
x^2w + 4xyw - 11xw^2 - y^3 - 3y^2z + 8y^2w - 3yz^2 + 16yzw -24yw^2 - z^3 + 7z^2w - 9zw^2 + 3w^3, \\
xz + 4xw - y^2 - 4yz + 9yw - 2z^2 + 3zw - w^2.
\end{cases}
$$
After a suitable coordinate change, the degree $2$ homogeneous equation can be written as:
$$-2x^2 + y^2 - 2z^2 + 445w^2=(y+\sqrt{2}x)(y-\sqrt{2}x)-(\sqrt{2}z+\sqrt{445}w)(\sqrt{2}z-\sqrt{445}w),$$
which is isomorphic the ruled surface $uv-st$ over $\Q(\sqrt{2},\sqrt{445})$. Thus $X_0(70)/\langle w_5\rangle$ is not trigonal over $\Q$. The models the quadratic surfaces for genus $4$ curves are given in Table \ref{quadratic surface for genus 4 quotient curves}. Since the curves $X_0(N)/\langle w_d\rangle$ always has a $\Q$-rational cusp,  form the discussions above we conclude that 
\begin{thm}
Suppose that $g(X_0(N)/\langle w_d\rangle)\geq 3$. Then $X_0(N)/\langle w_d\rangle$ is trigonal over $\Q$ if and only if $(N,w_d)$ is in the following list.
\begin{longtable}{|c|c|}
\hline
$g(X_0(N)/\langle w_d\rangle)$&$(N,w_d)$\\
\hline
3&$(42,w_{2}),
(42,w_{7}),
(57,w_{19}),
(58,w_{2}),
(65,w_{5}),
(65,w_{13}),
(77,w_{7}),
(82,w_{41}),$\\
&$(91,w_{13}),
(105,w_{35}),
(118,w_{59}),
(123,w_{41}),
(141,w_{47}),
$\\
\hline
4&$(66 ,w_{33}),
(74 ,w_{37}),
(86 ,w_{43})$\\
\hline
\end{longtable}
Consequently, for such pairs $(N,w_d)$ the set  $\Gamma_3'(X_0(N)/\langle w_d\rangle, \Q)$ is infinite. 
\end{thm}
Now consider the pairs $(N,w_d)$ such that $\mathrm{Gon}(X_0(N)/\langle w_d\rangle)=3$, but there is no $\Q$-rational degree $3$ mapping
$X_0(N)/\langle w_d\rangle\rightarrow \mathbb{P}^1$. A similar argument as in \cite[p. 352]{Jeo21} shows that if the set $\Gamma_3^\prime(X_0(N)/\langle w_d\rangle, \Q)$ is infinite, then $W_3(X_0(N)/\langle w_d\rangle)$ contains a translation of an elliptic curve $E$ with positive $\Q$-rank.
\begin{thm}
The set $\Gamma_3'(X_0(N)/\langle w_d\rangle, \Q)$ is finite for the following pairs of $(N,w_d)$:
\begin{align*}
&(66 ,w_{2}),
(70 ,w_{5}),
(74 ,w_{2}),
(77 ,w_{11}),
(82 ,w_{2}),
(85 ,w_{5}),
(85 ,w_{17}),
(91 ,w_{7}),
(93 ,w_{3}),
(110 ,w_{55}),\\
&(133 ,w_{19}),
(145 ,w_{29}),
(177 ,w_{59}).
\end{align*}
\end{thm}
\begin{proof}
For $N= 66, 70, 85, 93, 110, 133, 177$,
there is no
elliptic curve $E$ of positive $\Q$-rank with $\mathrm{cond}(E)\mid
N$. Hence in these cases, the set $\Gamma_3'(X_0(N)/\langle w_d\rangle, \Q)$ is finite. In the remaining cases, the Jacobian decomposition of $X_0(N)/\langle w_d\rangle$ are given by:
\begin{align*}
J_0(74)^{\langle w_2\rangle}&\sim^\Q 37a1\times 37b1 \times A_{f,\dim=2}\\
J_0(77)^{\langle w_{11}\rangle}&\sim^\Q 77a1\times 77b1 \times A_{f,\dim=2}\\
J_0(82)^{\langle w_2\rangle}&\sim^\Q 82a1 \times A_{f,\dim=3}\\
J_0(91)^{\langle w_7\rangle}&\sim^\Q 91a1 \times A_{f,\dim=3}\\
J_0(145)^{\langle w_{29}\rangle}&\sim^\Q 145a1 \times A_{f,\dim=3}.
\end{align*}
Note that in these cases, $X_0(N)/\langle w_d\rangle$ is bielliptic and there are
elliptic curves of positive $\Q$-rank with $\mathrm{cond}(E)|N$. By
arguments in \cite[Page 353]{Jeo21}, if there is no $\Q$-rational
degree $3$ mapping $X_0(N)/\langle w_d\rangle\rightarrow E$ where $E$ is an elliptic
curve of positive $\Q$-rank and $\mathrm{cond}(E)\mid N$, then
$W_3(X_0(N)/\langle w_d \rangle)$ has no translation of an elliptic curve with positive
$\Q$-rank. Form the Jacobian decomposition,  we only need to check whether triples 
(74,2,37a1), (77,11,77a1), (82,2,82a1), (91,7,91a1), (145,29,145a1) are admissible or not. 

Since $w_{37}$ acts as $+1$ on $37a1$ and $37a1$ has no non-trivial $2$-torsion over $\Q$, from Lemma \ref{descent} we conclude that the triple (74,2,37a1) is not admissible. In the remaining cases, if any of the triple $(N,d,E)$ is admissible (consequently, there is a $\mathbb{Q}$-rational degree $6$ mapping $X_0(N)\rightarrow E$), then the degree of the strong Weil parametrization of $E$ should divide $6$. For all the curves $77a1, 82a1,91a1$ and $145a1$ the degree of the strong Weil parametrization is $4$. Thus none of the triples is admissible. The result follows.
\end{proof}

\section{Bielliptic cases}
We are now left to discuss the pairs $(N,w_d)$ such that $X_0(N)/\langle w_d \rangle$ is bielliptic and $\mathrm{Gon}(X_0(N)/\langle w_d \rangle)> 3$. For the benefit of the readers, we recall such pairs  $(N,w_d)$:
\begin{longtable}{|c|c|}
\caption{Bielliptic remaining cases}
\label{Final remaining cases}
\\ \hline
$g(X_0(N)/\langle w_d\rangle)$&$(N,w_d)$\\
\hline
5&$(66,w_3),(66,w_{22}),(70,w_2),(70,w_7),(78,w_3),(86,w_2),(105,w_5),(105,w_{21}),$\\
&$(110,w_{11}),(111,w_3),(155,w_{31})$\\
\hline
6&$(78,w_2),(78,w_{13}),(111,w_{37}),(143,w_{13}),(145,w_5),(159,w_{53})$\\
\hline
7&$(105,w_3),(105,w_7),(105,w_{15}),(110,w_{10}),(118,w_2),(123,w_3),(143,w_{11})$\\
\hline
8&$(110,w_2),(110,w_5),(141,w_3),(155,w_5)$\\
\hline
9&$(142,w_2),(159,w_3).$\\
\hline
\end{longtable}
Following \cite[Page 353]{Jeo21}, for the above pairs $(N,w_d)$, if the set $\Gamma_3'(X_0(N)/\langle w_d\rangle, \Q)$ is infinite, then $W_3(X_0(N)/\langle w_d \rangle)$
contains a translation of an elliptic curve $E$ with positive
$\Q$-rank, equivalently the triple $(N,d,E)$ is admissible.
\begin{thm}
For the pairs $(N,w_d)$ in Table \ref{Final remaining cases}, the set $\Gamma_3'(X_0(N)/\langle w_d\rangle, \Q)$ is finite.
\end{thm}
\begin{proof}
For $N= 66, 70, 78, 105, 110$,
there is no
elliptic curve $E$ of positive $\Q$-rank with $\mathrm{cond}(E)\mid
N$. Hence for such values of $N$ and the corresponding values of $d$ as in Table \ref{Final remaining cases}, the set $\Gamma_3'(X_0(N)/\langle w_d\rangle, \Q)$ is finite.

For $N=118, 123, 141, 142, 143, 145, 155$, the only elliptic curves $E$ with positive $\Q$-rank has conductor equal to $N$. For such $N,E$ and the corresponding $d$ as in Table \ref{Final remaining cases}, if the triple $(N,d,E)$ is admissible, then the degree of the strong Weil parametrization of $E$ should divide $6$. From Cremona's table we see that for elliptic curves $E$ with positive $\Q$-rank of conductors $118, 123, 141, 142, 143, 145$ and $155$, the degree of the strong Weil parametrization of $E$ does not divide $6$. Consequently, for $N=118, 123, 141, 142, 143, 145, 155$, and the corresponding $d$ as in Table \ref{Final remaining cases}, the set $\Gamma_3'(X_0(N)/\langle w_d\rangle, \Q)$ is finite.

Finally, we are left to check whether the triples $(86,2,43a1),(111,3,37a1), (111,37,37a1)$ and $(159,3,53a1)$ are admissible or not. Since the elliptic curves $43a1,37a1$ and $53a1$ has no non-trivial $2$-torsion points, by Lemma \ref{descent} (\ref{descent 1}), we conclude that the triples $(86,2,43a1),(111,3,37a1)$ and $(159,3,53a1)$ are not admissible. Consequently, the sets $\Gamma_3'(X_0(86)/\langle w_2\rangle, \Q), \Gamma_3'(X_0(111)/\langle w_3\rangle, \Q)$ and $\Gamma_3'(X_0(153)/\langle w_3\rangle, \Q)$ are finite.

A similar argument as in \S\ref{Quadratic forms for the remaining cases section}, shows that if $\varphi: X_0(111)/\langle w_{37}\rangle\rightarrow 37a1$ is a $\Q$-rational mapping then we must have
\begin{equation}
\deg(\varphi)= 4x_1^2-6x_1x_2+4x_2^2, \ \mathrm{for\ some} \ x_1,x_2\in \Z.
\end{equation}
Thus $\deg(\varphi)$ can not take the value $3$. Consequently, the triple $(111,37,37a1)$ is not admissible. This completes the proof.
\end{proof}

\appendix
\section{Appendix}
Let $N\leq 623$. Suppose that $\mathrm{Gon}(X_0^{+d}(N))\geq 4$ and $X_0^{+d}(N)$ has no degree $\leq 2$ map to an elliptic curve. After applying Lemma \ref{sieve} (\ref{ii} and \ref{iii}), we are left to check the existence of $\Q$-rational degree $3$ mapping $X_0^{+d}(N)\rightarrow E$ where $E$ is an elliptic curve with positive $\Q$-rank, only for the following values of $N$:

106, 114, 122, 129, 130, 154, 158, 159, 166, 174, 178, 182, 183, 185,
195, 202, 215, 222, 231, 237, 246, 249, 258, 259, 262, 265, 267, 273, 282, 285,
286, 301, 303, 305, 326, 371, 393, 395, 407, 415, 427, 445, 473, 481.

\begin{center}

\begin{longtable}{|c|c|}
\caption{Hyperelliptic curve $X_0(N)/w_d$}
\label{Hyperelliptic values}
\\ \hline
$g(X_0(N)/\langle w_d\rangle)$&$(N,w_d)$\\
\hline
2&$(30,w_2),(30,w_3),(30,w_{10}),(33,w_3),(35,w_7),(38,w_2),(39,w_{13}),(42,w_3),$\\
 &$(42,w_6),(42,w_{21}),(57,w_3),(58,w_{29}),(66,w_{11}),(70,w_{35}),(78,w_{39}),$\\
 &$(87,w_{29}),(142,w_{71})$\\
 \hline
 3&$(46,w_2),(51,w_3),(55,w_5),(70,w_{14}),(78,w_{26}),(95,w_{19})$\\
 \hline
 4&$(62,w_2),(66,w_6),(69,w_3),(70,w_{10}),(119,w_{17})$\\
 \hline
 5&$(87,w_3),(95,w_5)$\\
 \hline
 6&$(78,w_6),(94,w_2),(119,w_7)$.\\
 \hline
\end{longtable}

\begin{longtable}{|c|c|}
\caption{$X_0(N)/\langle w_d \rangle$ with $\mathrm{Gon}(X_0(N)/\langle w_d \rangle)=3$}
\label{Trigonal values}
\\ \hline
$g(X_0(N)/\langle w_d\rangle)$&$(N,w_d)$\\
\hline
3&$(42,w_{2}),
(42,w_{7}),
(57,w_{19}),
(58,w_{2}),
(65,w_{5}),
(65,w_{13}),
(77,w_{7}),
(82,w_{41}),$\\
&$(91,w_{13}),
(105,w_{35}),
(118,w_{59}),
(123,w_{41}),
(141,w_{47}),
$\\
\hline
4&$(66 ,w_{2}),
(66 ,w_{33}),
(70 ,w_{5}),
(74 ,w_{2}),
(74 ,w_{37}),
(77 ,w_{11}),
(82 ,w_{2}),
(85 ,w_{5}),$\\
&$(85 ,w_{17}),
(86 ,w_{43}),
(91 ,w_{7}),
(93 ,w_{3}),
(110 ,w_{55}),
(133 ,w_{19}),
(145 ,w_{29}),
(177 ,w_{59})$\\
\hline
\end{longtable}

\begin{longtable}{|c|c|}
\caption{Bielliptic curve $X_0(N)/\langle w_d\rangle$}
\label{Bielliptic values}
\\ \hline
$g(X_0(N)/\langle w_d\rangle)$&$(N,w_d)$\\
\hline
2&$(30,w_2),(30,w_3),(30,w_{10}),(42,w_3),(42,w_6),(42,w_{21}),(57,w_3),(58,w_{29}),$\\
&$(66,w_{11}),(70,w_{35}),(78,w_{39}),(142,w_{71})$\\
\hline
3&$(42,w_2),(42,w_7),(57,w_{19}),(58,w_2),(65,w_5),(65,w_{13}),(70,w_{14}),(77,w_7),$\\
&$(78,w_{26}),(82,w_{41}),(91,w_{13}),(105,w_{35}),(118,w_{59}),(123,w_{41}),(141,w_{47})$\\
\hline
4&$(66,w_2),(66,w_{33}),(70,w_5),(70,w_{10}),(74,w_2),(74,w_{37}),(77,w_{11}),(82,w_2),$\\
&$(86,w_{43}),(91,w_7),(110,w_{55}),(145,w_{29})$\\
\hline
5&$(66,w_3),(66,w_{22}),(70,w_2),(70,w_7),(78,w_3),(86,w_2),(105,w_5),(105,w_{21}),$\\
&$(110,w_{11}),(111,w_3),(155,w_{31})$\\
\hline
6&$(78,w_2),(78,w_{13}),(111,w_{37}),(143,w_{13}),(145,w_5),(159,w_{53})$\\
\hline
7&$(105,w_3),(105,w_7),(105,w_{15}),(110,w_{10}),(118,w_2),(123,w_3),(143,w_{11})$\\
\hline
8&$(110,w_2),(110,w_5),(141,w_3),(155,w_5)$\\
\hline
9&$(142,w_2),(159,w_3).$\\
\hline
\end{longtable}

\begin{longtable}{|c|c|}
\caption{Models and Quadratic surface for $X_0(N)/\langle w_d \rangle$ with $g_{X_0(N)/\langle w_d \rangle}=4$}\label{quadratic surface for genus 4 quotient curves}
\\ \hline
Curve & Petri's model and Quadratic surface
\\ \hline
$X_0(66)/\langle w_2 \rangle$&$
\begin{cases}
12x^2w - 8xyw - 4xw^2 - 3y^3 + 11y^2z - 2y^2w + 6yz^2 - 2yzw + 3yw^2 - 72z^3\\
 &\hspace{-4cm}+ 18z^2w + 15zw^2 + 2w^3,\\
12xz - 8xw - 3y^2 + 8yz - 2yw - 12zw + w^2.
\end{cases}
$\\
&Diagonal form: $-12x^2 - 15y^2 + 540z^2 - 4428w^2$,\\
&lies on a ruled surface over $\Q(\sqrt{-41})$\\
\hline 
$X_0(66)/\langle w_{33} \rangle$&$
\begin{cases}
x^2z - xy^2 + 3y^2z - 2y^2w + 9yz^2 + 3yzw + 4yw^2 + 8z^3 + 9z^2w + 2zw^2 - 2w^3,\\
xw - yz - 2z^2 - 3zw.
\end{cases}$\\
&Diagonal form: $-x^2 - y^2 + z^2 + w^2$, lies on a ruled surface over $\Q$\\
\hline
$X_0(70)/\langle w_{5} \rangle$&$
\begin{cases}
x^2w + 4xyw - 11xw^2 - y^3 - 3y^2z + 8y^2w - 3yz^2 + 16yzw -24yw^2 - z^3 \\
&\hspace{-4cm} + 7z^2w - 9zw^2 + 3w^3,\\
xz + 4xw - y^2 - 4yz + 9yw - 2z^2 + 3zw - w^2.
\end{cases}
$\\
&Diagonal form: $-2x^2 + y^2 - 2z^2 + 445w^2$\\
&lies over a ruled surface over $\Q(\sqrt{2},\sqrt{445})$.\\
\hline
$X_0(74)/\langle w_{2} \rangle$&$
\begin{cases}
x^2w + xyw - 4xw^2 - y^3 - 4y^2z + 5y^2w - 10yz^2 + 7yzw - 8yw^2- 20z^3\\
&\hspace{-4cm} + 3z^2w - 2zw^2 + 6w^3,\\
xz + xw - y^2 - yz + 2yw - 4z^2 - 2zw - 2w^2.
\end{cases}
$\\
&Diagonal form: $15x^2 - 60y^2 - 4z^2 - 7w^2$,\\
&lies over a ruled surface over $\Q(\sqrt{-7})$.\\
\hline
$X_0(74)/\langle w_{37} \rangle$&$
\begin{cases}
x^2w + xyw + 4xw^2 - y^3 - 3y^2z - 2yz^2 + 3yw^2 - z^2w - 6zw^2 -2w^3,\\
xz + xw - y^2 - yz + 2yw - z^2 - w^2
\end{cases}
$\\
&Diagonal form: $3x^2 - 3y^2 - z^2 + w^2$, lies over a ruled surface over $\Q$.\\
\hline
$X_0(77)/\langle w_{11} \rangle$&$
\begin{cases}
2x^2w - 4xyw + 20xw^2 - y^3 + 3y^2z + 4y^2w - 3yz^2 - 12yzw +8yw^2 + z^3 \\
&\hspace{-4cm}+ 18z^2w + 44zw^2 + 8w^3,\\
xz - 4xw - y^2 + 2yz - 2yw - 2z^2 - 6zw - 2w^2
\end{cases}
$\\
&Diagonal form: $x^2 - 2y^2 - 2z^2 - 49w^2$, lies over a ruled surface over $\Q(i)$.\\
\hline
$X_0(82)/\langle w_{2} \rangle$&$
\begin{cases}
x^2w + 4xyw - 12xw^2 - 8y^3 - 24y^2z + 24y^2w - 44yz^2 + 68yzw - 8yw^2\\
\hspace{6.9cm} - 40z^3 + 128z^2w - 108zw^2 + 31w^3,\\
xz + xw - 2y^2 - 4yz - 8z^2 + 9zw - 3w^2.
\end{cases}
$\\
&Diagonal form $6x^2 - 24y^2 - 8z^2 - 18w^2$, lies over a ruled surface over $\Q(i)$.\\
\hline
$X_0(85)/\langle w_{5} \rangle$&$
\begin{cases}
72x^2w + 4xyw - 28xw^2 - 18y^3 + 25y^2z + 75y^2w - 81yz^2 -58yzw \\
\hspace{6cm}- 107yw^2  + 81z^3 + 252z^2w + 144zw^2 + 71w^3,\\
18xz + 2xw - 9y^2 - yz + 15yw - 18z^2 - 27zw - 16w^2
\end{cases}
$\\
&Diagonal form: $1886652x^2 - 11646y^2 - 18z^2 - 14623740w^2$,\\
&lies on a ruled surface over $\Q(\sqrt{2},\sqrt{-5015})$.\\
\hline
$X_0(85)/\langle w_{17} \rangle$&$
\begin{cases}
18x^2w - 6xyw + 40xw^2 - 9y^3 + 3y^2z + 12y^2w - 15yz^2 - 32yzw\\
\hspace{6cm}- 118yw^2 + 21z^3 + 80z^2w + 202zw^2 + 212w^3,\\
3xz - 2xw - 3y^2 + yz - yw - z^2 + 13zw + 2w^2
\end{cases}
$\\
&Diagonal form: $297x^2 - 11y^2 - z^2 + 29835w^2$\\
&lies on ruled surface over $\Q(\sqrt{3}, \sqrt{1105})$.\\
\hline
$X_0(86)/\langle w_{43} \rangle$&$
\begin{cases}
x^2z - xy^2 + y^2z - 2y^2w + 5yz^2 + 3yzw + 4yw^2 + 4z^3 + 4z^2w+ 2zw^2 - 2w^3,\\
xw - yz - z^2 - zw
\end{cases}
$\\
&Diagonal form: $-x^2 + 3y^2 - 3z^2 + w^2$, lies over a ruled surface over $\Q$.\\
\hline
$X_0(91)/\langle w_{7} \rangle$&$
\begin{cases}
72x^2w - 60xyw + 52xw^2 - 18y^3 + 57y^2z - 93y^2w + 75yz^2 -146yzw \\
 \hspace{6cm}+25yw^2 + 48z^3 - 109z^2w + 81zw^2 - 8w^3,\\
6xz - 10xw - 3y^2 + 5yz - 7yw + 4z^2 - 3zw - 4w^2
\end{cases}
$\\
&Diagonal form: $-7884x^2 - 292y^2 + 4z^2 + 6588w^2$,\\
&lies on a ruled surface over $\Q(\sqrt{-3},\sqrt{61})$.\\
\hline
$X_0(93)/\langle w_{3} \rangle$&$
\begin{cases}
4500x^2w - 1050xyw + 25xw^2 - 180y^3 + 30y^2z + 270y^2w - 96yz^2 \\
\hspace{4cm}+319yzw - 274yw^2 + 30z^3 + 563z^2w - 412zw^2 + 282w^3,\\
30xz - 35xw - 6y^2 + 7yz + 2yw - z^2 + 8zw - 6w^2
\end{cases}
$\\
&Diagonal form: $-5400x^2 + 25y^2 - z^2 + 1171800w^2$\\
&lies on a ruled surface over $\Q(\sqrt{6},\sqrt{217})$.\\
\hline
$X_0(110)/\langle w_{55} \rangle$&$
\begin{cases}
x^2w - xyw + xw^2 - y^3 + y^2w + 3yzw + yw^2 + z^2w + zw^2 - w^3,\\
xz - xw - y^2 + yz + 2yw + 2zw - w^2
\end{cases}
$\\
&Diagonal form: $-x^2 - 3y^2 + 3z^2 + 13w^2$, lies on a ruled surface over $\Q(\sqrt{13})$\\
\hline
$X_0(133)/\langle w_{19} \rangle$&$
\begin{cases}
54x^2w - 9xyw - 18xw^2 - 6y^3 - 3y^2z + 3y^2w - 4yz^2 + 10yzw -7yw^2\\
\hspace{8cm} - 4z^3 + 16z^2w - 16zw^2 + 12w^3,\\
6xz - 3xw - 2y^2 + yz - yw + 2zw - 3w^2
\end{cases}
$\\
&Diagonal form: $-6x^2 - 282y^2 + 3384z^2 - 504w^2$,\\
&lies on a ruled surface over $\Q(\sqrt{3},\sqrt{-7})$.\\
\hline
$X_0(145)/\langle w_{29} \rangle$&$
\begin{cases}
x^2w - 2xyw - xw^2 - y^3 + y^2z - 2y^2w + 5yz^2 - 4yzw + 2z^3 -2z^2w,\\
xz - 2xw - y^2 + 2yz - yw + 2z^2 - 3zw + w^2
\end{cases}
$\\
&Diagonal form: $-3x^2 - 6y^2 + 2z^2 + 21w^2$,\\
&lies on a ruled surface over $\Q(\sqrt{3},\sqrt{7})$.\\
\hline
$X_0(177)/\langle w_{59} \rangle$&$
\begin{cases}
x^2w - xw^2 - y^3 - y^2z - yz^2 + w^3,\\
xz - y^2 - yw - zw - w^2
\end{cases}
$\\
&Diagonal form: $-x^2 - y^2 + z^2 - 3w^2$, lies on a ruled surface over $\Q(\sqrt{-3})$.\\
\hline
\end{longtable}
\end{center}

\bibliographystyle{alpha}

\noindent{Francesc Bars Cortina}\\
{Departament Matem\`atiques, Edif. C, Universitat Aut\`onoma de Barcelona\\
08193 Bellaterra, Catalonia}\\
{Francesc.Bars@uab.cat}

\vspace{1cm}

\noindent{Tarun Dalal}\\
 {Institute of Mathematical Sciences, ShanghaiTech University\\ 393 Middle Huaxia Road, Pudong, Shanghai 201210, China}\\
 {tarun.dalal80@gmail.com}

 

\end{document}